\documentclass[12]{amsart}
\usepackage{amsmath,amssymb,amsthm,color,enumerate,comment,centernot,enumitem,url,cite}
\usepackage{graphicx,relsize,bm}
\usepackage{mathtools}
\usepackage{array}

\makeatletter
\newcommand{\tpmod}[1]{{\@displayfalse\pmod{#1}}}
\makeatother

\newtheorem{thm}{Theorem}[section]
\newtheorem{lemma}[thm]{Lemma}

\theoremstyle{remark}

\theoremstyle{definition}

\newtheorem{rem}[thm]{Remark}

\theoremstyle{THM}

\newcommand{\abs}[1]{\left|{#1}\right|}

\def\FF {{\mathcal F}}
\def\GG {{\mathcal G}}

\def\R {{\mathbb R}}
\def\Z {{\mathbb Z}}

\def\Q {{\mathbb Q}}

\def\GG {{\mathcal G}}

\def\F {{\mathbb F}}

\def\Z {{\mathbb Z}}
\def\Q {{\mathbb Q}}

\def\Gal{{\mbox{{\rm{Gal}}}}}

\makeatletter
\@namedef{subjclassname@2020}{%
  \textup{2020} Mathematics Subject Classification}
\makeatother

\def\red#1 {\textcolor{red}{#1 }}
\def\blue#1 {\textcolor{blue}{#1 }}

\numberwithin{equation}{section}

\def\Z {{\mathbb Z}}

\begin{document}

\title[Monogenic Cyclic Cubic Trinomials]{Monogenic Cyclic Cubic Trinomials}

\author{Lenny Jones}
\address{Professor Emeritus, Department of Mathematics, Shippensburg University, Shippensburg, Pennsylvania 17257, USA}
\email[Lenny~Jones]{doctorlennyjones@gmail.com}

\date{\today}

\begin{abstract}
 A series of recent articles has shown that there exist only three monogenic cyclic quartic trinomials in ${\mathbb Z}[x]$, and they are all of the form $x^4+bx^2+d$. In this article, we conduct an analogous investigation for cubic trinomials in ${\mathbb Z}[x]$. Two irreducible cyclic cubic trinomials are said to be \emph{equivalent} if their splitting fields are equal. We show that there exist two infinite families of non-equivalent monogenic cyclic cubic trinomials of the form $x^3+Ax+B$. We also show that there exist exactly four monogenic cyclic cubic trinomials of the form $x^3+Ax^2+B$, all of which are equivalent to  $x^3-3x+1$. 
\end{abstract}


\maketitle
\section{Introduction}\label{Section:Intro}

A monic polynomial $f(x)\in \Z[x]$ of degree $n$ that is irreducible over $\Q$ is called \emph{cyclic} if the Galois group over $\Q$ of $f(x)$, denoted $\Gal(f)$, is the cyclic group $C_n$ of order $n$, while $f(x)$ is called \emph{monogenic} if $\{1,\theta,\theta^2,\ldots, \theta^{n-1}\}$ is a basis for the ring of integers of $\Q(\theta)$, where $f(\theta)=0$. Hence, $f(x)$ is monogenic if and only if  $\Z_K=\Z[\theta]$. For the minimal polynomial $f(x)$ of an algebraic integer $\theta$ over $\Q$, it is well known \cite{Cohen} that
\begin{equation} \label{Eq:Dis-Dis}
\Delta(f)=\left[\Z_K:\Z[\theta]\right]^2\Delta(K),
\end{equation}
where $\Delta(f)$ and $\Delta(K)$ are the respective discriminants over $\Q$ of $f(x)$ and the number field $K$.
Thus, from \eqref{Eq:Dis-Dis}, $f(x)$ is monogenic if and only if $\Delta(f)=\Delta(K)$.

In a series of recent articles \cite{HJF1,JonesBAMSEven,JonesQuartic} it was shown that the only monogenic cyclic quartic trinomials are
\begin{equation*}\label{Eq:Complete}
x^4-4x^2+2,\quad x^4+4x^2+2\quad \mbox{and} \quad x^4-5x^2+5.
\end{equation*}
In this article, our focus is on an investigation of the analogous question for monogenic cyclic cubic trinomials.
    We say that two monogenic cyclic cubic trinomials are \emph{equivalent} if their splitting fields are equal. Otherwise, we say they are \emph{distinct}. Trivially, a monogenic cyclic cubic trinomial is equivalent to itself.
 Our main result is:
\begin{thm}\label{Thm:Main}
  Let $A,B\in \Z$ with $AB\ne 0$.
  \begin{enumerate}
    \item \label{Main:I1} There exist two infinite single-parameter families, $\FF_1$ and $\FF_2$, of monogenic cyclic cubic trinomials of the form $x^3+Ax+B$, such that $\FF_1 \cap \FF_2=\varnothing$ and no two elements of $\FF_1 \cup \FF_2$ are equivalent. 
    \item \label{Main:I2} There are exactly four monogenic cyclic trinomials of the form $x^3+Ax^2+B$:
    \[x^3+3x^2-3,\quad x^3-3x^2+3,\quad x^3-3x^2+1, \quad x^3+3x^2-1;\] and they are all equivalent to $T(x)=x^3-3x+1\in \FF_1$.
  \end{enumerate}
\end{thm}

Monogenic cyclic cubic fields have been investigated by various authors \cite{A,Gras1,Gras2,KS}. Gras \cite{Gras1,Gras2} gave three characterizations of monogenic cyclic cubic fields  in terms of the solvability of certain Diophantine equations. More recently, Kashio and Sekigawa \cite{KS} showed that every monogenic cyclic cubic field is the splitting field of $x^3-tx^2-(t+3)x-1$ for some $t\in \Z$, and they gave a description of this characterization in terms of the parameter $t$. Most of these characterizations involve the conductor of the field. Theorem \ref{Thm:Main} represents a new and different approach to the characterization of monogenic cyclic cubic fields with a specific focus on trinomials and their particular forms.
\begin{rem}
  The polynomials $x^3-tx^2-(t+3)x-1$ were first studied by Daniel Shanks \cite{Shanks}, and their splitting fields are known as the \emph{simplest cubic fields}.
\end{rem}

\section{Preliminaries}\label{Section:Prelim}
The formula for the discriminant of an arbitrary monic trinomial, due to Swan \cite[Theorem 2]{Swan}, is given in the following theorem. 
\begin{thm}
\label{Thm:Swan}
Let $f(x)=x^n+Ax^m+B\in \Q[x]$, where $0<m<n$, and let $d=\gcd(n,m)$. Then $\Delta(f)=$
\[
(-1)^{n(n-1)/2}B^{m-1}\left(n^{n/d}B^{(n-m)/d}-(-1)^{n/d}(n-m)^{(n-m)/d}m^{m/d}A^{n/d}\right)^d.
\]
\end{thm}
\noindent
Note that if $f(x)$ is a cubic trinomial, then $f(x)$ is cyclic if and only if $\Delta(f)$ is a square.

The next result, due to Jakhar, Khanduja and Sangwan \cite[Theorem 1.1]{JKS2}, is helpful in determining the monogenicity of a trinomial. We use the notation $q^e\mid \mid W$ to mean that $q^e$ is the exact power of the prime $q$ that divides the integer $W$.
\begin{thm}{\rm \cite{JKS2}}\label{Thm:JKS}
Let $N\ge 2$ be an integer.
Let $K=\Q(\theta)$ be an algebraic number field with $\theta\in \Z_K$, the ring of integers of $K$, having minimal polynomial $f(x)=x^{N}+Ax^M+B$ over $\Q$, with $\gcd(M,N)=d_0$, $M=M_1d_0$ and $N=N_1d_0$. A prime factor $q$ of $\Delta(f)$ does not divide $\left[\Z_K:\Z[\theta]\right]$ if and only if 
$q$ satisfies one of the following conditions:
\begin{enumerate}[font=\normalfont]
  \item \label{JKS:I1} when $q\mid A$ and $q\mid B$, then $q^2\nmid B$;
  \item \label{JKS:I2} when $q\mid A$ and $q\nmid B$, then
  \[\mbox{either } \quad q\mid A_2 \mbox{ and } q\nmid B_1 \quad \mbox{ or } \quad q\nmid A_2\left((-B)^{M_1}A_2^{N_1}-\left(-B_1\right)^{N_1}\right),\]
  where $A_2=A/q$ and $B_1=\frac{B+(-B)^{q^j}}{q}$ with $q^j\mid\mid N$;
  \item \label{JKS:I3} when $q\nmid A$ and $q\mid B$, then
  \begin{gather*}
   \mbox{either} \quad q\mid A_1 \quad \mbox{and}\quad  q\nmid B_2\\
   \mbox{or}\\
    q\nmid A_1B_2^{M-1}\left((-A)^{M_1}A_1^{N_1-M_1}-\left(-B_2\right)^{N_1-M_1}\right),
  \end{gather*}
  where $A_1=\frac{A+(-A)^{q^\ell}}{q}$ with $q^\ell\mid\mid (N-M)$, and $B_2=B/q$;
  \item \label{JKS:I4} when $q\nmid AB$ and $q\mid M$ with $N=s^{\prime}q^k$, $M=sq^k$, $q\nmid \gcd\left(s^{\prime},s\right)$, then the polynomials
   \begin{equation*}
     H_1(x):=x^{s^{\prime}}+Ax^s+B \quad \mbox{and}\quad H_2(x):=\dfrac{Ax^{sq^k}+B+\left(-Ax^s-B\right)^{q^k}}{q}
   \end{equation*}
   are coprime modulo $q$;
         \item \label{JKS:I5} when $q\nmid ABM$, then
     \[q^2\nmid \left(B^{N_1-M_1}N_1^{N_1}-(-1)^{M_1}A^{N_1}M_1^{M_1}(M_1-N_1)^{N_1-M_1} \right).\]
   \end{enumerate}
   \end{thm}

Although the following lemma is little more than an observation, it will be useful enough to state formally.
\begin{lemma}\label{Lem:Equiv}
 Let $f(x)$ and $g(x)$ be monogenic cyclic cubic trinomials. Then $f(x)$ and $g(x)$ are equivalent if and only if $\Delta(f)=\Delta(g)$.
  \end{lemma}
  \begin{proof} Let $K$ be the splitting field of $f(x)$. Then the lemma is an immediate consequence of the fact that because $f(x)$ is monogenic, it follows  from  \eqref{Eq:Dis-Dis} that $\Delta(f)=\Delta(K)$.
  \end{proof}
\section{The Proof of Theorem \ref{Thm:Main}}\label{Section:MainProof}
\begin{proof}
  For item \eqref{Main:I1}, let $k\in \Z$. Define
  \begin{gather}
   \label{F1} f_{1,k}(x):=x^3-3\delta_1 x+(6k+1)\delta_1\\ \nonumber
    \mbox{and}\\
    \label{F2} f_{2,k}(x)=x^3-\delta_2 x+(2k+1)\delta_2,
  \end{gather}
  where
  \[\delta_1:=9k^2+3k+1 \quad \mbox{and} \quad \delta_2:=27k^2+27k+7.\] Then, we have from Theorem \ref{Thm:Swan} that
  \begin{equation}\label{Eq:Dis}
  \Delta(f_{1,k})=3^4\delta_1^2 \quad \mbox{and} \quad  \Delta(f_{2,k})=\delta_2^2.
  \end{equation} Observe that
  \[f_{1,k}(x)\equiv f_{2,k}(x)\equiv x^3+x+1 \pmod{2}.\] Hence, since $x^3+x+1$ is irreducible in $\F_2[x]$, we deduce that $f_{i,k}(x)$ is irreducible over $\Q$. Therefore, since $\Delta(f_{i,k})$ is a perfect square, it follows that $\Gal(f_{i,k})\simeq C_3$.

  We claim that $f_{i,k}(x)$ is monogenic if and only if $\delta_i$ is squarefree. We give details only for $i=1$ since the approach is similar for $i=2$. Let $\Z_K$ denote the ring of integers of $K=\Q(\theta)$, where $f_{1,k}(\theta)=0$.  Suppose first that $q^2\mid \delta_1$ for some prime $q$. Then, it is easy to see that condition \eqref{JKS:I1} of Theorem \ref{Thm:JKS} is not satisfied, so that $q\mid [\Z_K:\Z[\theta]]$ and $f_{1,k}(x)$ is not monogenic.
  Conversely, suppose that $\delta_1$ is squarefree, and let $q$ be a prime divisor of $\delta_1$. Note that $q\not \in \{2,3\}$. If $q\mid (6k+1)$, then $k\equiv -1/6 \pmod{q}$ and
  \[\delta_1 \equiv 9(-1/6)^2+3(-1/6)+1\equiv 3/4 \not \equiv 0 \pmod{q}.\] Hence, $q^2\nmid (6k+1)\delta_1$ so that condition \eqref{JKS:I1} of Theorem \ref{Thm:JKS} is satisfied. Thus, $q\nmid [\Z_K:\Z[\theta]]$. We must still examine condition \eqref{JKS:I2} of Theorem \ref{Thm:JKS} for the prime $q=3$.  In this case we have that 
  \[B=-(6k+1)\delta_1, \quad A_2=-\delta_1 \quad \mbox{and}\quad B_1=\frac{(6k+1)\delta_1-(6k+1)^3\delta_1^3}{3}.\] As previously noted, $3\nmid A_2$ so that the first possibility in condition \eqref{JKS:I2} of Theorem \ref{Thm:JKS} is not satisfied. Using Maple to examine the second possibility in condition \eqref{JKS:I2} yields
  \[(-B)^{M_1}A_2^{N_1}-\left(-B_1\right)^{N_1}=(6k+1)\delta_1^4+\left(\frac{(6k+1)\delta_1-(6k+1)^3\delta_1^3}{3}\right)^3\equiv 1 \pmod{3},\]
  which confirms that this second possibility of condition \eqref{JKS:I2} is indeed satisfied. Consequently, $3\nmid [\Z_K:\Z[\theta]]$ and $f_{1,k}(x)$ is monogenic, which establishes the claim.

  We can now define two families of monogenic trinomials:
  \begin{gather}\label{Eq:Fams}
  \FF_1:=\{f_{1,k}(x):k\in \Z, \ \mbox{with $\delta_1$ squarefree}\}\\ \nonumber
   \mbox{and}\\
   \FF_2:=\{f_{2,k}(x):k\in \Z, \ \mbox{with $k\ge 0$ and $\delta_2$ squarefree}\}.
  \end{gather} For each $i\in \{1,2\}$, it follows from a result of Nagel \cite{Nagel} that there exist infinitely many $k\in \Z$ such that $\delta_i$ is squarefree. Hence, each of the families $\FF_i$ is an infinite set.

   If $f_{1,k}(x)=f_{2,m}(x)$ for some $k,m\in \Z$, then $\Delta(f_{1,k})=\Delta(f_{2,m})$, so that
   \[k=\frac{3\pm \sqrt{3(108m^2+108m+19)}}{18}\in \Z,\] which is impossible since $108m^2+108m+19\equiv 1\pmod{3}$. Hence, $\FF_1\cap \FF_2=\varnothing$.

   We show next that no two elements of $\FF_1$ are equivalent. By way of contradiction, assume that for some $f_{1,k}(x),f_{1,m}(x)\in \FF_1$ with $k\ne m$ and $f_{1,k}(\alpha)=f_{1,m}(\beta)=0$, we have $\Q(\alpha)=\Q(\beta)$. Since $f_{1,k}(x)$ and $f_{1,m}(x)$ are both monogenic, it follows from Lemma \ref{Lem:Equiv} and \eqref{Eq:Dis} that
  \begin{equation}\label{Eq:1}
  3^4(9k^2+3k+1)=3^4(9m^2+3m+1).
  \end{equation} Using Maple to solve \eqref{Eq:1} with the restriction that $k\ne m$, we get three solutions. The first solution has $k=-m-1/3$, which contradicts the fact that $k,m\in \Z$. The remaining two solutions require that $-36m^2-12m-7$ is the square of an integer. However, it is easy to see that $-36m^2-12m-7<0$ for all $m\in \R$. Therefore, the elements of $\FF_1$ are distinct.

  Turning to $\FF_2$, we apply the same strategy to show that no two elements of $\FF_2$ are equivalent, and we use Maple to solve
  \[(27k^2+27k+7)^2=(27m^2+27m+7)^2,\] where $k\ne m$. In this situation, Maple again gives three solutions. The first solution has $k=-m-1$, which is impossible since $km\ge 0$. The remaining two solutions require that $-324m^2-324m-87$ is the square of an integer, which is impossible since $m\ge 0$. Hence, the elements of $\FF_2$ are distinct.

  We show now that no element of $\FF_1$ is equivalent to some element of $\FF_2$. Using the same approach as before, we assume that
    \begin{equation}\label{Eq:3}
  3^4(9k^2+3k+1)^2=(27m^2+27m+7)^2,
  \end{equation} for some $k,m\in \Z$ with $m\ge 0$. Solving \eqref{Eq:3} using Maple gives four solutions. Two of these solutions require that $-108m^2-108m-55$ is the square of an integer, which is impossible since $m\ge 0$. A third solution has
  \[k=\frac{-3+\sqrt{108m^2+108m+1}}{18},\] which contradicts the fact that $k\in \Z$ since $108m^2+108m+1\equiv 1 \pmod{3}$. The fourth solution has
  $k=(-3-\sqrt{108m^2+108m+1})/18$ and so we arrive at the same contradiction. Thus, we have established the fact that the trinomials in $\FF_1\cup \FF_2$ are all distinct.

  For item \eqref{Main:I2}, let $f(x)=x^3+Ax^2+B$ be a monogenic cyclic trinomial.
  By Theorem \ref{Thm:Swan}, we have that
   \begin{equation}\label{Eq:Delf}
   \Delta(f)=-B(4A^3+27B), 
    \end{equation} which must be a square since $f(x)$ is cyclic. It is straightforward to verify that $\Delta(f)>1$, and we omit the details. We let $q$ be a prime divisor of $\Delta(f)$, 
    and we apply  Theorem \ref{Thm:JKS} to obtain necessary conditions on $A$ and $B$ for the monogenicity of $f(x)$.

    Suppose that $\abs{B}>1$ and that $q\mid B$. If $q\mid A$, then since condition \eqref{JKS:I1} of Theorem \ref{Thm:JKS} must be satisfied, we conclude that $q^2\nmid B$. Hence, $q\mid \mid B$. If $q\nmid A$, then since $N-M=1$, we have that $\ell=0$ and $A_1=0$ in condition \eqref{JKS:I3} of Theorem \ref{Thm:JKS}. Thus, $q\mid A_1$, and so $q^2\nmid B$ since $f(x)$ is monogenic. Hence, again we see that $q\mid\mid B$. Therefore, we have shown that $B$ is squarefree.

Suppose next that $q\nmid A$ and $q\nmid B$. If $q=2$, then we see from condition \eqref{JKS:I4} of Theorem \ref{Thm:JKS} that $H_2(x)=0$. Thus, $H_1(x)$ and $H_2(x)$ are not coprime modulo $q=2$, contradicting the fact that $f(x)$ is monogenic. Hence, $q\ne 2$ and $q\nmid 2AB$. Then, from condition \eqref{JKS:I5} of Theorem \ref{Thm:JKS}, we have that $q^2\nmid (4A^3+27B)$, which contradicts the fact that $\Delta(f)$ is a square. It follows that any prime divisor $q$ of $\Delta(f)$ must divide either $A$ or $B$.

Suppose then that $q\mid A$ and $q\nmid B$. Thus, we see from \eqref{Eq:Delf} that $q=3$. Observe if $9\mid A$, then it follows from \eqref{Eq:Delf} that $3^3\mid\mid \Delta(f)$, which contradicts the fact that $\Delta(f)$ is a square. Therefore, $3\mid \mid A$.

    We provide the following summary to emphasize three facts that we have gleaned from Theorem \ref{Thm:JKS}:
    \begin{gather}
      \label{Info:I1} \mbox{$B$ is squarefree,}\\
      \label{Info:I2} q\mid A \quad \mbox{or} \quad q\mid B,\\
      \label{Info:I0} \mbox{if $q\mid A$ and $q\nmid B$, then $q=3$ and $3\mid \mid A$,}
     \end{gather}
         where $q$ is a prime divisor of $\Delta(f)$.
Since $\Delta(f)$ is a square, we deduce from \eqref{Eq:Delf} and \eqref{Info:I1} that $B\mid (4A^3+27B)$. Hence, $B\mid 2A$, so that $B^2\mid 4A^2$. Thus,
 \begin{gather}
  \nonumber \Delta(f)=B^2Z, \quad \mbox{where} \\ 
 \label{Eq:Z} Z=-\left(\frac{4A^3}{B}+27\right)=-\left(\left(\frac{4A^2}{B^2}\right)AB+27\right)\in \Z
 \end{gather}
 is a square.
Note that if $Z=1$, then $-A^3=7B$, which is impossible by \eqref{Info:I1}. Hence, $Z>1$ and we let $q$ be a prime divisor of $Z$. It follows from \eqref{Info:I2} and \eqref{Eq:Z} that $q=3$. Therefore, $Z=3^{2k}$ for some integer $k\ge 1$, and we have from \eqref{Eq:Z} that
\begin{equation}\label{Eq:3^{2k}}
B=\frac{-4A^3}{3^{2k}+27}\in \Z. 
\end{equation}
Next, we examine the solutions to \eqref{Eq:3^{2k}} when $k\in \{1,2\}$.


When $k=1$, we get the equation $-(A/3)^3=B/3$ from \eqref{Eq:3^{2k}} and, by \eqref{Info:I1}, the solutions are easily seen to be
\[(A,B)\in \{(3,-3), (-3,3)\}.\]
It is straightforward to verify that both of the trinomials
    \begin{equation}\label{Eq:k=1}
    x^3+3x^2-3 \quad \mbox{and} \quad x^3-3x^2+3
    \end{equation}
    are monogenic and cyclic.

    When $k=2$, we arrive at the equation $-(A/3)^3=B$ from \eqref{Eq:3^{2k}}, and it is easy to see that the solutions are
    \[(A,B)\in \{(-3,1), (3,-1)\}\] by \eqref{Info:I1}. As before, it is straightforward to confirm that both of the trinomials
       \begin{equation}\label{Eq:k=2}
       x^3-3x^2+1 \quad \mbox{and}\quad x^3+3x^2-1
       \end{equation} are monogenic and cyclic.

       We assume now that $k\ge 3$. It is easy to see that $3\mid A$ in \eqref{Eq:3^{2k}}, and since $2^2\mid \mid (3^{2k}+27)$, we can rewrite \eqref{Eq:3^{2k}} as
       \begin{gather}\label{Eq:3^{2k} again}
       \begin{split}
        BW=\left(\frac{-A}{3}\right)^3,\\ 
        \mbox{where \ $W=\dfrac{3^{2k-3}+1}{4}\in \Z$.}
        \end{split}
       \end{gather} 
       If $3\mid B$, then $3\mid \mid B$ by \eqref{Info:I1}, which implies that $3\mid \mid BW$ since $3\nmid W$, contradicting the fact that $BW$ is a cube  in \eqref{Eq:3^{2k} again}. Thus, $3\nmid B$ and we conclude from \eqref{Info:I0} that $3\mid \mid A$.
         We therefore examine condition \eqref{JKS:I2} of Theorem \ref{Thm:JKS} with $q=3$.  Straightforward calculations with $B$ as in \eqref{Eq:3^{2k}} show that
       \[A_2=A/3 \quad \mbox{and} \quad B_1=\frac{4A_2^3\left(16A_2^6-(3^{2k-3}+1)^2\right)}{3(3^{2k-3}+1)^3}.\] Since $3\nmid A_2$ and
       \[\frac{16A_2^6-(3^{2k-3}+1)^2}{3}\equiv 2 \pmod{3},\]
       it follows that
       \begin{equation*}
         (-B)^2A_2^3+B_1^3=\frac{16A_2^9\left((3^{2k-3}+1)^7+4\left(\frac{16A_2^6-(3^{2k-3}+1)^2}{3}\right)^3\right)}{(3^{2k-3}+1)^9}\equiv 0 \pmod{3},
       \end{equation*} which proves that $f(x)$ is not monogenic by Theorem \ref{Thm:JKS}. Hence, the four monogenic cyclic trinomials given in \eqref{Eq:k=1} and \eqref{Eq:k=2} are the only monogenic cyclic trinomials of the form $x^3+Ax^2+B$.

              To see that the trinomials in \eqref{Eq:k=1} and \eqref{Eq:k=2}
       are equivalent to $T(x)=x^3-3x+1$, suppose that $T(\alpha)=0$. Then $T(x)$ factors over $\Q(\alpha)$ as
       \[T(x)=x^3-3x+1=(x-\alpha)(x-\alpha^2+2)(x+\alpha^2+\alpha-2),\]
        and it is easy to verify that the four trinomials in \eqref{Eq:k=1} and \eqref{Eq:k=2} factor over $\Q(\alpha)$ as:
       \begin{align*}
         x^3+3x^2-3&=(x+\alpha+1)(x+\alpha^2-1)(x-\alpha^2-\alpha+3),\\
         x^3-3x^2+3&=(x-\alpha-1)(x-\alpha^2+1)(x+\alpha^2+\alpha-3),\\
         x^3-3x^2+1&=(x+\alpha-1)(x+\alpha^2-3)(x-\alpha^2-\alpha+1),\\
         x^3+3x^2-1&=(x-\alpha+1)(x-\alpha^2+3)(x+\alpha^2+\alpha-1).\qedhere
       \end{align*}
        \end{proof}

\section{Final Remarks}
Theorem \ref{Thm:Main} gives two infinite mutually exclusive single-parameter families $\FF_1$ and $\FF_2$ of monogenic cyclic cubic trinomials of the form $x^3+Ax+B$. However, these families are not exhaustive. That is, there exist other monogenic cyclic cubic trinomials of the form $x^3+Ax+B$ that are not contained in $\FF_1\cup \FF_2$ and not equivalent to any element of $\FF_1\cup \FF_2$. We do not know if these trinomials can somehow be parameterized into one or more infinite families, or whether they are isolated, and we leave such a task for future research. We end by providing some examples of these trinomials:
\begin{gather*}
  x^3-6447x+199243,\quad  x^3-23907x+1422773,\quad  x^3-66063x+6535601,\\
  x^3-123411x+16687025, \quad x^3-1738191x+882052345,\\
  x^3-47970741x+127882981837.
\end{gather*}






\end{document}